\documentclass[12pt]{amsart}
\usepackage{amsfonts,amsmath,amssymb}
\usepackage{hyperref}
\usepackage{amssymb,amsthm,amsxtra}
\usepackage[usenames]{color}
\usepackage{amscd}
\usepackage{amsthm}
\usepackage{amsfonts}
\usepackage{amssymb}
\usepackage{amsmath}
\usepackage{graphicx}
\usepackage{enumerate}
\providecommand{\U}[1]{\protect\rule{.1in}{.1in}}

\newtheorem*{theorem*}{Theorem}
\newtheorem*{lemma*}{Lemma}
\newtheorem{lemma}{Lemma}[subsection]
\newtheorem{proposition}[lemma]{Proposition}
\newtheorem{remark}[lemma]{Remark}

\newtheorem{theorem}[lemma]{Theorem}

\newtheorem{corollary}[lemma]{Corollary}

\newtheorem*{conjecture*}{Conjecture}
\newtheorem*{remark*}{Remark}
\newtheorem{thm}[lemma]{Theorem}

\newtheorem{defn}[lemma]{Definition}

\newtheorem{cor}[lemma]{Corollary}

\newtheorem{rem}[lemma]{Remark}

\newtheorem{introtheorem}{Theorem}

\newtheorem{introthm}[introtheorem]{Theorem}

\baselineskip 18pt \textwidth 16cm  \theoremstyle{plain}

\newcommand{\dseq}{\mathbf{d}}

\newcommand{\depth}{\operatorname{depth}}

\newcommand{\bs}{\backslash}

\newcommand{\Hom}{\operatorname{Hom}}

\newcommand{\diag}{\operatorname{diag}}

\newcommand{\cc}{\mathbb{C}}
\newcommand{\qq}{\mathbb{Q}}
\newcommand{\rr}{\mathbb{R}}

\newcommand{\A}{\mathbb{A}}

\newcommand{\triv}{{\bf 1}}

\newcommand{\Ind}{\operatorname{Ind}}

\newcommand{\re}{\operatorname{Re}}

\newcommand{\Q}{{\mathbb Q}}
\newcommand{\R}{{\mathbb R}}

\newcommand{\pr}{{\operatorname{pr}}}

\newcommand{\model}{{\mathcal{M}}}

\newcommand{\Fre}{{Fr\'{e}chet \,}}

\newcommand{\abs}[1]{\left|{#1}\right|}

\newcommand{\aut}{{\operatorname{aut}}}

\newcommand{\lieg}{\operatorname{\mathfrak{g}}}
\newcommand{\Ad}{\operatorname{Ad}}

\newcommand{\vv}{\mathcal{V}}

\newcommand{\udg}{\widehat{G_n}}
\newcommand{\ud}[1]{\widehat{#1}}

\title{Existence of Klyachko models for $GL(n,\rr)$ and $GL(n,\cc)$}
\date\today

\author{Dmitry Gourevitch}
\address{Dmitry Gourevitch, Faculty of Mathematics and Computer Science, Weizmann
Institute of Science, POB 26, Rehovot 76100, Israel }
\email{dimagur@weizmann.ac.il}
\urladdr{\url{http://www.wisdom.weizmann.ac.il/~dimagur}}

\author{Omer Offen}
\address{Omer Offen, Department of Mathematics,
Technion-Israel Institute of Technology, Technion City, Haifa 32000, Israel}
\email{offen@tx.technion.ac.il}
\urladdr{\url{http://www.technion.ac.il/~offen/}}

\author{Siddhartha Sahi}
\address{Siddhartha Sahi, Department of Mathematics, Rutgers University, Hill Center -
Busch Campus, 110 Frelinghuysen Road Piscataway, NJ 08854-8019, USA}
\email{sahi@math.rugers.edu}

\author{Eitan  Sayag}
\address{Eitan  Sayag,  Department of Mathematics,
Ben Gurion University of the Negev,
P.O.B. 653
Be'er Sheva 84105, Israel}
\email{eitan.sayag@gmail.com}
\urladdr{\url{http://www.math.bgu.ac.il/~sayage/}}

\begin{document}
\begin{abstract}
We prove that any unitary representation of $GL(n,\rr)$ and $GL(n,\cc)$ admits an
 equivariant linear form with respect
to one of the subgroups  considered by Klyachko.
\end{abstract}

\maketitle

\tableofcontents

\section{Introduction}

Let $F$ be either ${\mathbb{R}}$ or ${\mathbb{C}}$.
Let $G_{n}:={GL}(n,F)$ and let $\widehat{G_n}$ denote the unitary dual of $G_n$.
For $\pi\in \widehat{G_n}$ let $\pi^{\infty}$ denote the \Fre space of smooth vectors in $\pi$.
For any decomposition $n=r+2k$ we
consider a subgroup of $G_n$ defined by
\[
H_{r,2k}=\left\{  \left(
\begin{array}
[c]{cc}%
u & X\\
0 & h
\end{array}
\right)  \in G_n:u\in N_{r},\,X\in M_{r\times 2k}(F)\text{ and }h\in
Sp(2k)\right\}  .
\]
Here $N_{r} \subset G_r$ denotes the group of $r\times r$ upper unitriangular
matrices and
\begin{equation} \label{eq: symp form}
Sp(2k)=\left\{  g\in G_{2k}:{}^{t}g J_k  g=J_k  \right\} \text{ where } J_k = \left(
\begin{array}
[c]{cc}
& w_{k}\\
-w_{k} &
\end{array}
\right)
\end{equation}
and $w_{k}\in G_{k}$ is the permutation matrix with $(i,j)$th
entry equal to $\delta_{k+1-i,j}$. Let $\psi$ be a non-trivial additive
character of $F$. We associate to $\psi$ the character $\psi_r$ of
$N_r$ defined by
\[
\psi_r(u)=\psi(u_{1,2}+\cdots+u_{r-1,r})
\]
and the character $\phi_{r,2k}$ of $H_{r,2k}$
defined by
$$
\phi_{r,2k}\left(
\begin{array}
[c]{cc}%
u & X\\
0 & h
\end{array}
\right)  =\psi_{r}(u).
$$

%The space of invariant linear functionals associated with the

For $\pi\in\widehat{G_{n}}$ we consider the space
$\Hom_{H_{r,2k}}(\pi^{\infty},\phi_{r,2k})$ of continuous
$(H_{r,2k},\phi_{r,2k})$-equivariant linear forms on the Frech\'et
space $\pi^{\infty}$ of smooth vectors in $\pi$. We refer to a
non-zero element of  $\Hom_{H_{r,2k}}(\pi^{\infty},\phi_{r,2k})$
as a \emph{Klyachko linear form} of type $(r,2k)$.
Let
\[
\model_{r,2k}=\{f:G_n\rightarrow\cc: f \text{ is smooth and }f(hg)=\phi_{r,2k}(h)f(g),\ h\in H_{r,2k},\,g\in G_n\}.
\]
If $\pi$ is an irreducible Hilbert representation of $G_n$ then a non-zero element $\ell\in\Hom_{H_{r,2k}}(\pi^{\infty},\phi_{r,2k})$ defines a realization of $\pi^\infty$ in the space of functions $\model_{r,2k}$
via $v\mapsto f_v:\pi^\infty\rightarrow \model_{r,2k}$ where $f_v(g)=\ell(\pi(g)v),\,g\in G$.
We therefore refer to $\model_{r,2k}$ as the \emph{Klyachko model} of type $(r,2k)$.
With this relation in mind for the rest of this paper we focus on Klyachko linear forms rather then Klyachko models.
%We note
%that by Frobenuis reciprocity (Theorem 5.3.3.1 of \cite{MR0498999}) we
%obtain an embedding of $\pi^{\infty}$ into an appropriate space of
%functions (see \cite{MR0498999} section 5.1), which we call the
%\emph{Klyachko model.}

An analogue of this finite family of spaces of linear forms
associated with representations of $GL(n)$ over a finite field was
first considered by Klyachko \cite{Kly}, followed by the work of
Inglis-Saxl \cite{MR1129515} and the work of Howlett-Zworestine
\cite{MR1802175}.

Over a $p$-adic field, the problem was first considered by
Heumos-Rallis \cite{MR1078382} and further studied by Offen-Sayag in \cite{MR2332593, MR2417789, MR2414223, MR2515933}.
The outcome is disjointness and uniqueness of Klyachko linear forms (in fact for all irreducible
admissible representations) and existence for any representation in the unitary dual.
In fact, the type of Klyachko linear form that a given unitary representation $\pi$ admits is made explicit as follows.
Based on the classification of Tadic for the unitary dual \cite{MR870688} Venkatesh associated in \cite{MR2133760} a partition $\vv(\pi)$, the $SL(2)$-type of $\pi$,  to any unitary $\pi$. We recall the classification and the definition of the $SL(2)$-type in Section \ref{subsec:ClassPart} below.
If $r$ is the number of odd parts of the partition $\vv(\pi)$ then $\Hom_{H_{r,2k}}(\pi^{\infty},\phi_{r,2k})\ne 0$ (see \cite[(5.1)]{MR2515933}).

The existence of a Klyachko linear form in the $p$-adic case is proved along the following lines. In the case $r=0$ a linear form invariant by the symplectic group is constructed by a global (automorphic) argument for building blocks of unitary representations (namely, generalized Speh representations) \cite[Proposition 1]{MR2332593}. For a general unitary representation associated with the case $r=0$, the invariant linear form is obtained by an explicit construction for induced representations based on  Bernstein's principle of meromorphic continuation \cite[Proposition 2]{MR2332593}.
The general case, treated in \cite{MR2417789}, is obtained by a reduction to the case $r=0$ using the theory of derivatives of Bernstein-Zelevinsky \cite{MR0579172}.

In this work we address the existence of Klyachko linear forms over Archimedean fields.
Disjointness and uniqueness are addressed in an upcoming work \cite{AOS}. The partition $\vv(\pi)$ for every $\pi\in \widehat{G_n}$ is defined in \cite[\S 2.2]{MR2133760} also in the Archimedean case.
%in terms of the interpretation due to Tadic \cite[Theorem D]{MR870688} of Vogan's classification \cite{MR827363}.
Our main result is
\begin{introthm}\label{thm:main}
Let $\pi\in\widehat{G}_{n}$ and let $r$ be the number of odd parts of the partition $\vv(\pi)$. Then
$\operatorname{Hom}_{H_{r,n-r}}(\pi^{\infty}, \phi_{n-r,r})
\ne 0$.
\end{introthm}

The scheme of proof in the $p$-adic case, described above, serves us as a guideline to prove the Theorem.
For the case $r=0$ a global argument similar to the $p$-adic case treats Speh representations and an explicit construction for induced representations is based on the work of Carmona-Delorme \cite{MR1274587}.

However, a theory of derivatives is not available in this context.
In order to reduce to the case $r=0$ we apply the theory of highest derivatives developed by Sahi in \cite{MR1026329} for unitary representations. A recent result of Gourevitch-Sahi (\cite[Theorem B]{GS}) allows us to express the integer $r$ in terms of the partition corresponding to $\pi$.

%\subsection{Structure of the paper}

%In \S \ref{sec:Prel} we give the necessary preliminaries. In \S \ref{subsec:Part} we define the associated variety and associated partition and describe them for all unitary representations of $G_n$ in terms of the Vogan classification, which we also recall. In \S\S \ref{subsec:DerDepth} we recall the definition of highest derivative (or adduced representation) from \cite{Sahi-Kirillov} and the depth sequence introduced recently in \cite{GS}.

%In \S \ref{sec:Sp} we prove Theorem \ref{thm:main} for the case $r=0$. In this case $n$ is even and the theorem just states existence of an $Sp(n)$-invariant functional. Since $(G_n,Sp(n))$ is a symmetric pair, this reduces (as we show in \S\S \ref{subsec:Prod}) to existence of an invariant functional on each block of the Vogan classification.
%For the case $F=\C$ this finishes the argument, since all characters of $G_k$ are trivial on $Sp(k)$. For $F=\R$ one also has to consider the Speh representation. This is done in \S\S \ref{subsec:Speh} using a global (automorphic) argument. We would like to remark that an $Sp(k,\bR)$-invariant functional on the Speh representation is also constructed in \cite{GSS} in a more explicit local way, thus giving a purely local construction of the Archimedean Klyachko models.

%In \S \ref{sec:PfMain} we prove Theorem \ref{thm:main} by reducing it to the case $r=0$ using derivatives.
\section{Preliminaries}\label{sec:Prel}

\subsection{Smooth vectors and induction}
\label{subsec:Analvec}

Let $(\pi,\mathcal{H})$ be a continuous Hilbert representation of a Lie
group $G$. A vector $v\in \mathcal{H}$ is called smooth if the map $g\mapsto\pi(g)v:G\rightarrow \mathcal{H}$ is infinitely differentiable. Both $G$
and its Lie algebra ${\mathfrak{g}}$ act on the space of smooth
vectors in $\mathcal{H}$ and we denote the corresponding representation by $(
\pi^{\infty},\mathcal{H}^{\infty})$.
It is naturally a \Fre representation of $G$.

\begin{thm}[Harish-Chandra] \label{thm:UniSmoothIrr}
Let $(\pi,\mathcal{H})$ be a unitary representation of a real reductive group $G$. Then $\pi$ is irreducible if and only if $\pi^{\infty}$ is irreducible. (cf. \cite[Theorem 3.4.11]{MR929683}).
\end{thm}
\begin{remark}
In fact [loc. cit.] says that $\pi$ is irreducible if and only if $\pi_K$, the underlying $(\mathfrak{g},K)$ module with respect to a compact subgroup $K$ of $G$, is irreducible. Since a $G$-invariant decomposition of $\pi$ (resp.~$\pi^\infty$) clearly provides one of $\pi^\infty$ (resp.~$\pi_K$), the above Theorem is indeed straightforward from [loc. cit.].
\end{remark}

Let $G$ be a Lie group with a Lie algebra $\mathfrak{g}$. Denote by $\Delta_{G}:G\rightarrow \rr_{>0}$ the modular function associated with $G$, i.e.
\[
    \Delta_G(g)=\abs{\det(\Ad(g)_{|\lieg})}.
\]

Let $H$ be a closed subgroup of $G$, $(\sigma,V)$ a Hilbert representation of $H$ and $\delta:H\rightarrow \rr_{>0}$ defined by $\delta\left(
h\right)  =\Delta_{H}\left(  h\right)  /\Delta_{G}\left(  h\right)$.

Let $W$ denote the Hilbert space of equivalence classes of
measurable functions $f:G\rightarrow V$ such that
\[
f(hg) =\delta^{\frac12}(
h)  \sigma(h)f(g)\ \ \  \text{ and } \ \ \  \left\Vert f\right\Vert^2_W:=\int_{H\bs G} \left\Vert f\left(
g\right)  \right\Vert _{V}^{2}\ dg<\infty.
\]
Let $(\pi,W)$ be the representation of $G$ defined by $\pi(g)f(x)=f(xg),\ x,\,g\in G$.
Denote the representation $(\pi,W)$ by $\Ind_{H}^{G}(\sigma)$, the normalized induction of $\sigma$ from $H$ to $G$.
If $(\sigma,V)$ is unitary then $\Ind_{H}^{G}(\sigma)$ is also unitary.

Recall the following result of Poulsen. It can be interpreted as a
representation-theoretic version of Sobolev's embedding theorem.
\begin{thm}[see \cite{MR0310137}, Theorem 5.1]\label{thm:Pou}
Let $(\sigma,V)$ be a unitary representation of $H$ and let $(\pi,W)=\Ind_{H}^{G}(\sigma)$.
Then $\Ind_{H}^{G}(\sigma)^{\infty}$ consists of all infinitely differentiable functions $f \in W$ such that all their derivatives with respect to left-$G$-invariant differential operators on $G$ are square integrable.
\end{thm}

We will apply Poulsen's Theorem for certain Hilbert representations induced from a one dimensional twist of a unitary representation.
For the rest of this section let $\chi$ be a (not necessarily unitary) character of $H$ that extends to a smooth function $\chi':G\rightarrow \cc^*$.
Let $(\sigma,V)$ be a unitary representation of $H$ and $(\pi,W)=\Ind_H^G(\sigma)$.

There is an isomorphism of Hilbert representations $(\pi_\chi,W)\simeq \Ind_H^G(\sigma\otimes\chi)$ given by $f\mapsto \chi' f$, $f\in W$ where $\pi_\chi(g) f(x) =\chi'(x)^{-1}\chi'(xg) f(xg)$, $g,\,x\in G$.
Since $\chi'$ is smooth it follows that $w\in W$ is smooth with respect to $\pi_\chi$ if and only if it is smooth with respect to $\pi$. The following Corollaries are therefore immediate consequence of Poulsen's Theorem.

\begin{corollary}\label{cor: smth ind twist}
Every element of $\Ind_H^G(\sigma\otimes\chi)^\infty$ is an infinitely differentiable function on $G$ with values in $V$.
\end{corollary}
%\begin{remark}
%In fact, Poulsen's Theorem is only stated in [loc. cit.] under the
%assumption that $(\sigma,V)$ is unitary. The invariance of the form on % unitary structure of
%$\sigma$ is never used in the proof and the more general statement
%holds. Furthermore, our usage of Theorem \ref{thm:Pou}
%will be to certain non unitary representations that are twists of
%unitary representation by characters. This slight generalization
%of Poulsen's theorem follows easily from the original statement.
%\end{remark}
%This makes evaluation at points well defined on our spaces of induced representations.
%The image of evaluation at points in fact lies in the relevant space of smooth functions.
%\begin{cor}\label{cor: smth image}
% If $f\in
%\Ind_H^G(\sigma\otimes \chi)^{\infty}$ then $f(g)\in V^\infty,\ g\in G$.
%\end{cor}
%\begin{proof}
%Note that for $f\in \Ind_H^G(\sigma\otimes \chi)^{\infty}$ we have $(\chi')^{-1}f\in W^\infty$. Without loss of generality we may therefore assume that $\chi=\triv_H$. Let $f\in W^{\infty}$ and $g\in G$. Set $v=f(g)\in V$, then the action map of
%$H$ on $v$ is
%\[
%h\mapsto\sigma\left(  h\right)  v=\sigma\left(  h\right)  f\left(  g\right)
%=\delta^{-1/2}(h)f(hg).
%\]
%By Theorem \ref{thm:Pou} $f$ is a smooth function and therefore the action map $h\mapsto \sigma(h)v$ is smooth, i.e. $v\in V^\infty$.
%\end{proof}
%If $H\bs G$ is compact then the square-integrability condition in Theorem \ref{thm:Pou} is automatically satisfied.
%We therefore also have
\begin{cor}\label{cor:ParInd}
Suppose that $H
\backslash G$ is compact and let $f\in \Ind_H^G(\sigma\otimes \chi)$. Then  $f\in \Ind_H^G(\sigma\otimes \chi)^{\infty}$ if and only if
$f:G\rightarrow V$ is an infinitely differentiable function.
\end{cor}
%\begin{remark}
%Note that $f\in W^\infty$ is infinitely differentiable as a function with values in the Hilbert space $V$ if and only if
%it is infinitely differentiable as a function with values in the \Fre space $V^\infty$ (see \cite[Proposition 1.2]{MR0310137}).
%\end{remark}

\subsection{Representations of $GL(n)$-notation}
Let $F$ be either $\rr$ or $\cc$ and let $G_n=GL(n,F)$. Let $K=K_n$ be the standard maximal compact subgroup of $G_n$,
i.e. $O(n)$ if $F=\rr$ and $U(n)$ if $F=\cc$.

For a Hilbert representation $(\pi,V)$ of $G_n$ and $s\in \cc$ we denote by $(\abs{\ }^s\pi,V)$ the Hilbert representation on
the same space $V$ given by $g\mapsto \abs{\det g}^s\pi(g)$.

Let $(n_1,\dots,n_k)$ be a decomposition of $n$ and let $P=MU$ be the standard parabolic subgroup of $G_n$ consisting
of matrices in upper triangular block form, where
\[
M=\{\diag(m_1,\dots,m_k):m_i\in G_{n_i},\,i=1,\dots,k\}
\]
is the standard Levi subgroup of $P$ and $U$ is its unipotent radical.
Let $(\sigma_i,V_i)$ be a Hilbert representation of $G_{n_i}$, $i=1,\dots,k$ and let
$(\sigma,V)=(\sigma_1\otimes \cdots \otimes \sigma_k,V_1\otimes \cdots V_k)$ be the associated Hilbert representation of $M$. We also view $(\sigma,V)$ as a representation of $P$ where $U$ acts trivially. We use the following standard notation for normalized parabolic induction to $G_n$
\[
    \sigma_1 \times\cdots\times \sigma_k=\Ind_P^{G_n}(\sigma).
\]
For $\varphi\in \Ind_P^{G_n}(\sigma)$ let
\[
    \varphi_s(g)=\left[\prod_{i=1}^k \abs{\det m_i}^{s_i} \right]\varphi(g), \ \ \ g=umk\in G_n, u\in U, m=\diag(m_1,\dots,m_k)\in M,\, k\in K_n.
\]
We further associate to $\sigma$ a family $I(\sigma,s)$ of induced representations
parameterized by $s=(s_1,\dots,s_k)\in \cc^k$ realized in the underlying vector space of $\Ind_P^{G_n}(\sigma)$.
The representation $I(\sigma,s)$ is defined by
\[
    (I(g,\sigma,s)\varphi)_s(x)=\varphi_s(xg), \ \ \ \varphi\in\Ind_P^{G_n}(\sigma),\ g,\,x\in G_n.
\]
We have
\[
    I(\sigma,s)\simeq \abs{\ }^{s_1}\sigma_1 \times \cdots \times \abs{\ }^{s_k}\sigma_k
\]
 and the underlying space for $I(\sigma,s)^\infty$ is independent of $s$ (as explained in Section \ref{subsec:Analvec}).

\subsection{The unitary dual of $GL(n)$ and the $SL(2)$-type} \label{subsec:ClassPart}
The unitary dual $\udg$ of $G_n$ was classified by Vogan in \cite{MR827363}.
In \cite{MR870688}, Tadic classified the unitary dual of $GL(n)$ over a $p$-adic field
and expressed the classification in a uniform language for both the Archimedean and non-Archimedean cases.
We recall the classification as it appears in \cite[Theorem D]{MR870688}. (As noted in [ibid.] Tadic' Theorem D is also valid in the Archimedean case, see also \cite{MR2537046}.)

Let $\delta\in \ud{G_r}$ be square-integrable (thus $r=1$ if $F=\cc$ and $r\in\{1,2\}$ if $F=\rr$).
Denote by $U(\delta,t)$ the unique irreducible
quotient of
\[
|\ |^{\frac{t-1}2}\delta \times |\ |^{\frac{t-3}2}\delta \times\cdots \times |\ |^{\frac{1-t}2}\delta
\]
and for $\alpha\in \rr$,
$\abs{\alpha}<\frac12$ let
\[
\pi(\delta,t,\alpha)=|\ |^{\alpha}U(\delta,t) \times
|\ |^{-\alpha} U(\delta,t).
\]
For $r=1$ the representation $U(\delta,t)$ is one dimensional. For $r=2$ it was constructed in \cite{MR695900} using the
theory of automorphic forms. Later it was given an explicit
Hilbert space model in \cite{MR1062967}.

Let $B$ be the set of all representations of the form $U(\delta,t)$ or $\pi(\delta,t,\alpha)$ as above.
Then for any $\pi_1,\dots,\pi_k\in B$ the representation $\pi_1\times\cdots\times \pi_k\in\udg$ for an appropriate $n$ and any $\pi\in \udg$ is of this form for a uniquely determined multi-set $\{\pi_1,\dots,\pi_k\}$
in $B$.

In particular, for any $\pi\in \ud{G_n}$ there exist integers $k_1,\dots,k_m$, $t_1,\dots,t_m$, square integrable representations $\delta_i \in \ud{G_{k_i}}$ and $-\frac12<\alpha_i<\frac12$ such that
\[
    \pi=\abs{\det}^{\alpha_1}U(\delta_1,t_1) \times \cdots \times \abs{\det}^{\alpha_m}U(\delta_m,t_m).
\]
The following is therefore immediate from \cite[Proposition 1.9]{MR1026752}.
\begin{lemma}\label{lemma: gen irr}
Let $\pi_i\in \ud{G_{n_i}}$, $i=1,2$. Then the set
\[
\{s\in \cc: \pi_1 \times \abs{\det}^s \pi_2 \text{ is reducible}\}
\]
is discrete in $\cc$.
\end{lemma}

A partition of $n$ is a multi-set of positive integers adding up to $n$. By abuse of notation we will sometimes denote a partition $\lambda$ as a tuple $(n_1,\dots,n_k)$ but we keep in mind that order is irrelevant.
The integers $n_1,\dots,n_k$ are referred to as the parts of $\lambda$.
The transpose partition $\lambda^t$ is the partition $(m_1,\dots,m_l)$ where $m_i=\#\{j: 1\le j\le k,\,i\le n_j\}$ ($l$ is the maximal integer so that $\{j: 1\le j\le k,\,l\le n_j\}$ is not empty). If $\lambda$ and $\mu$ are partitions their union (as a multi-set)
is denoted by $(\lambda,\mu)$.
We call a partition even if all its parts are even and odd if all
its parts are odd.
For two natural numbers $r$ and $n$ let
\[\left\langle n\right\rangle_{r}=\overbrace{(n,\dots,n)}^{r}
\]
be the partition of $nr$ with $r$ equal parts.

The $SL(2)$-type associated to $\pi\in\udg$ is denoted by $\vv(\pi)$ and characterized by the following properties.
For any $\delta\in \ud{G_r}$  square integrable, $0 < \alpha <\frac12$, $\pi_1\in \ud{G_{n_1}}$ and $\pi_2\in \ud{G_{n_2}}$ we have
\begin{enumerate}
\item $\vv(U(\delta,n))=\left\langle n\right\rangle_r$;
\item $\vv(\pi(U(\delta,n),\alpha)))=\left\langle n\right\rangle_{2r}$;
\item $\vv(\pi_1\times \pi_2)=(\vv(\pi_{1}),\vv(\pi_{2}))$.
\end{enumerate}

\begin{defn}
A representation $\pi\in \udg$ is called even if
$\vv(\pi)$ is even and odd if $\vv(\pi)$ is odd. We denote by
$r(\pi)$ the number of odd parts in $\vv(\pi)$.
\end{defn}
Note that a product of two even representations is even. The following statement is straightforward from the definitions and the classification of $\udg$.
\begin{cor}
\label{cor:DecompEvenOdd} Let $\pi\in\udg$.
There is a decomposition $n=k+l$, $k,\,l \ge 0$,
$\pi_{e}\in\widehat{G}_{k}$ an even representation and $\pi_{o} \in\widehat{G}_{l}$ an odd representation, uniquely determined up to isomorphism, such that $\pi=
\pi_e\times\pi_o$.
\end{cor}

\subsection{The highest derivative}
\label{subsec:DerDepth}

The following convention will be used whenever convenient. For $n<m$ we view $G_n$ as a subgroup
of $G_m$ through the imbedding $g\mapsto \diag(g,I_{m-n})$. This convention will freely be used throughout the paper
for subgroups of $G_n$ without further notice.

For subgroups $A_i$ of $G_{k_i}$, $i=1,2$, by $(A_1 \times A_2) \ltimes M_{k_1\times k_2}(F)$ we mean the subgroup
of $G_{k_1+k_2}$ consisting of matrices of the form
\[
\diag(a_1,a_2)\ltimes X:=\begin{pmatrix}
a_1 & X \\ 0 & a_2
\end{pmatrix}, \ \ \ a_i\in A_i,\,i=1,2,\,X\in M_{k_1\times k_2}(F).
\]
In accordance with our convention, when $A_2=\{e\}$ we also set $A_1 \ltimes  M_{k_1\times k_2}(F)=
(A_1 \times A_2) \ltimes M_{k_1\times k_2}(F)$.

For a representation $(\sigma,V)$ of $A_1\times A_2$ and a character $\chi$ of $M_{k_1\times k_2}(F)$
we denote by $(\sigma\ltimes \chi,V)$ the representation of $(A_1 \times A_2)\ltimes M_{k_1\times k_2}(F)$ defined by
\[
(\sigma \ltimes \chi)(\diag(a_1,a_2)\ltimes X)=\chi(X) \sigma(\diag(a_1,a_2)),\ \ \ a_i\in A_i,\,i=1,2,\,X\in M_{k_1\times k_2}(F).
\]

We recall the Archimedean analog, as formulated in
\cite{MR1026329}, of the Bernstein-Zelevinsky notion of highest derivative
\cite{MR0579172}.

%\label{subsec:Mirab}

Denote by $P_n$ the \emph{\textquotedblleft
mirabolic\textquotedblright} subgroup of $G_n$ consisting of matrices with last row
$e_n:=(0,0,...,0,1)$, i.e. $P_n=G_{n-1} \ltimes F^{n-1}$.
Note that
\[
\Delta_{P_n}(g)=\abs{\det g},\,g\in P_n.
\]

The starting point of the Archimedean theory of highest derivatives is the following
\begin{thm}
\label{BarKir} Let $\pi\in\widehat{G_{n}}$, then $\pi|_{P_{n}}$ is irreducible.
\end{thm}

\begin{remark*}
The result was conjectured by Kirillov.
In the p-adic case it was proved in \cite{MR748505}, in the
complex case in \cite{MR1026329} and finally in the real case in
\cite{MR1999922}.
\end{remark*}

For a Hilbert representation $(\sigma,V)$ of $G_n$ let $E(\sigma)=\sigma\ltimes \triv_{F^n}$ be the associated representation of $P_{n+1}$ on the same space $V$.

For a Hilbert representation $(\tau,V)$ of $P_n$ let
\[
I(\tau)=\Ind_{P_n \ltimes F^n}^{P_{n+1}}(\tau\ltimes\psi(e_n \cdot)).
\]

Note that $E|_{\ud{G_n}}: \ud{G_n} \rightarrow \ud{P_{n+1}}$ and $I|_{\ud{P_n}}: \ud{P_n}\rightarrow \ud{P_{n+1}}$.

Based on Theorem \ref{BarKir}  and Mackey theory Sahi observed that
for $\pi\in \ud{G_n}$ there exists a unique integer $d$, $1\le d \le n$ and a unique $\sigma\in \ud{G_{n-d}}$ such that
\[
\pi|_{P_n}\simeq I^{d-1}E(\sigma).
\]
The representation $\sigma$ is called the \emph{highest derivative} (or \emph{adduced}) of $\pi$
and is denoted by $A(\pi)$.
The integer $d$ is called the \emph{depth} of $\pi$ and we denote it by $\depth(\pi)$.

Recursively we define $A^{j+1}(\pi)=A(A^j(\pi))$ as long as $A^j(\pi)$ is a representation of $G_i$ for some integer $i\ge 1$.
Let $k$ be such that $A^k(\pi)$ is the trivial representation of $G_0$.
The \emph{depth sequence} of $\pi$ is defined to be
\begin{equation}
\dseq(\pi)=(d_1,\dots,d_k) \text{
where }d_{j+1}=\depth(A^j\pi)  \label{=d_pi},\ j=0,\dots,k-1.
\end{equation}

The following Theorem follows from \cite[Theorem B]{GS}.
\begin{theorem}%[\cite{GS}, Theorem A]
\label{thm:GSMain}
Let $\pi\in \ud{G_n}$ and $\dseq(\pi)=(d_1,\dots,d_k)$ then $d_1 \ge \cdots \ge d_k$ and viewed as a partition $\dseq(\pi)$ satisfies
\begin{equation}\label{eq: transpose partition}
\vv(\pi)=\dseq(\pi)^t.
\end{equation}
\end{theorem}
\begin{corollary}\label{cor:DerOddEven}
Let $\pi\in \ud{G_n}$. Then
\begin{enumerate}
\item $\depth(\pi)$ is the number of parts in $\vv(\pi)$.
In particular, $\depth(\pi)\ge r(\pi)$ and equality holds if and only if $\pi$ is odd.
\item If $\pi$ is odd then
$A(\pi)$ is even.
\end{enumerate}
\end{corollary}
\begin{proof}
We use the notation of the Theorem. It is clear that $d_1$ is the number of parts in $\dseq(\pi)^t$.
Since by definition $d_1=\depth(\pi)$ the first part follows from \eqref{eq: transpose partition}. It follows from the definitions that $\dseq(A(\pi))=(d_2,\dots,d_k)$. Applying \eqref{eq: transpose partition} again we obtain that
$\vv(A(\pi))=\dseq(A(\pi))^t$ consists of parts of the form $m-1$ where $1<m$ is a part of $\dseq(\pi)^t=\vv(\pi)$. The second part follows.
\end{proof}

Let $n=m+r$. For Hilbert representations $\pi$ of $G_m$ and $\tau$ of $P_r$ we set
\[
\pi\times \tau=\Ind_{(G_m \times P_r) \ltimes M_{m\times r}(F)}^{P_n}((\pi \otimes \tau) \ltimes \triv_{M_{m\times r}(F)}).
\]

%We adjust to the smooth setting some properties of the maps $E$ and $I$ provided in \cite{MR1026329}.
\begin{lemma}\label{lem: sahi}
Let $s\in \cc$ and consider the Hilbert representations $\pi$ of $G_m$, $\sigma$ of $G_r$ and $\tau$ of $P_r$.
We have
\begin{enumerate}
\item \label{Es}$E(\abs{\ }^s \pi)=\abs{\ }^s E(\pi)$;
\item \label{Is}$I(\abs{\ }^s\tau)\simeq \abs{\ }^s I(\tau)$;
\item \label{Eprod}$E(\pi \times \sigma)=\pi \times E(\sigma)$;
\item \label{Iprod}$I(\pi \times \tau)=\pi \times I(\tau)$.
\end{enumerate}
\end{lemma}
\begin{proof}
Part \eqref{Es} is straightforward. Indeed, the underlying representation space of both $E(\abs{\ }^s \pi)$ and $\abs{\ }^s E(\pi)$ is that of $\pi$ and the two actions are identical.
For part \eqref{Is} set $f_s(p)=\abs{\det p}^s f(p)$, $p \in P_n$. The map $f\mapsto f_s$ is an isomorphism from
$\abs{\ }^s I(\tau)$ to $I(\abs{\ }^s\tau)$.
Parts \eqref{Eprod} and \eqref{Iprod} are proved in \cite[Lemma 2.1 (ii) and (iii)]{MR1026329} when $\pi$, $\sigma$ and $\tau$ are unitary. The proof of [ibid.] is valid verbatim in the more general context of Hilbert representations.

\end{proof}

Let $S_{m,r}$ be the subgroup of $G_n$ defined by
$S_{m,r}=(G_m\times N_r) \ltimes M_{m\times r}(F)$.

\begin{proposition}\label{prop: proj k sigma}
Let $d \le n$, $Q=MU$ a standard parabolic subgroups of $G_{n-d}$ with its standard Levi decomposition
($M\simeq G_{m_1}\times\cdots\times G_{m_k}$), $\tau$ a non-zero unitary representation of $M$, $s\in \cc^k$ and $(\sigma,V)=\Ind_Q^{G_{n-d}}(\tau,s)$. Let $\pi=I^{d-1}E(\sigma)$.
\begin{enumerate}[(i)]
\item \label{trans ind} We have $\pi\simeq \Ind_{S_{n-d,d}}^{P_n}((\sigma \otimes \psi_d)\ltimes \triv_{M_{n-d\times d}(F)})$.
\item \label{equiv prop}There is a continuous linear transformation
$\pr_{d,\sigma}:\pi^\infty \rightarrow V^\infty$ that is not identically zero on any $P_n$-invariant subspace of $\pi^\infty$ and satisfies
\begin{equation}\label{eq: k equiv cond}
    \pr_{d,\sigma}(\pi(s)v)=\psi_d(n)\abs{\det g}^{\frac{d-1}2}\sigma(g) \pr_{d,\sigma}(v),\ \ \ v\in \pi^\infty \text{ and } s=
    \begin{pmatrix} g & X \\ 0  & n \end{pmatrix}\in S_{n-d,d}
\end{equation}
where $g\in G_{n-d},\,n\in N_d$ and $X\in M_{n-d \times d}(F)$.

%In particular, if $\sigma$ is irreducible then $\pr_{d,\sigma}$ has dense image.

\end{enumerate}
\end{proposition}

\begin{proof}
Part \eqref{trans ind} follows by iteratively applying transitivity of induction.
For part \eqref{equiv prop} note that $V^\infty$, the space of smooth vectors for $\sigma$, is also the space of smooth vectors of
the representation $(\sigma \otimes \psi_k)\ltimes \triv_{M_{n-d\times d}(F)}$ of $S_{n-d,d}$.
Let
\[
\tau_1=(\tau \otimes \psi_d)\ltimes \triv_{M_{n-d\times d}(F)}
\]
be a unitary representation of the subgroup $Q_1:=(Q\times N_d) \ltimes M_{n-d\times d}(F)$ of $S_{n-d,d}$ and let
\[
\chi_1=(\chi_s \times \triv_{N_d})\ltimes \triv_{M_{n-d\times d}(F)}
\]
be a character of $Q_1$
where $\chi_s$ is the unramified character of $Q$ associated to $s$ by
\[
\chi_s(\diag(g_1,\dots,g_k)u)=\prod_{i=1}^k \abs{\det g_i}^{s_i},\ \ \ g_i\in G_{m_i},\,i=1,\dots,k,\,u\in U.
\]
It follows from Corollary \ref{cor: smth ind twist}  that the elements of $\Ind_{Q_1}^{P_n}(\tau_1 \otimes \chi_1)^\infty$ are smooth functions on $P_n$ with values in the space of $\tau$.
Let $\delta_1=\Delta_{P_n}/\Delta_{S_{n-d,d}}$ then transitivity of induction gives the isomorphism
\[
f\mapsto \varphi_f: \Ind_{Q_1}^{P_n}(\tau_1 \otimes \chi_1)^\infty \rightarrow\Ind_{S_{n-d,d}}^{P_n}((\sigma \otimes \psi_d)\ltimes \triv_{M_{n-d\times d}(F)})^\infty
\]
where $\varphi_f(p)(s)=\delta_1^{\frac12}(s)f(sp)$, $s\in S_{n-d,d}$, $p\in P_n$.
Since $f$ is a smooth function on $P_n$ it now follows that $\varphi_f$ is a smooth function on $G$ with values in $V$. It further follows from Corollary \ref{cor:ParInd} that $\varphi_f(p)\in V^\infty$ for $p\in P_n$.

To summarize so far, the elements of $\Ind_{S_{n-d,d}}^{P_n}((\sigma \otimes \psi_d)\ltimes \triv_{M_{n-d\times d}(F)})^\infty$ are smooth functions on $P_n$ with values in $V^\infty$.

Thus, $\pr_{d,\sigma}(\varphi):= \varphi(e)$ is a well defined linear transformation from $\Ind_{S_{n-d,d}}^{P_n}((\sigma \times \psi_k)\ltimes \triv_{M_{n-d\times d}(F)})^\infty$ to $V^\infty$. Evaluation at the identity is clearly not identically zero on any $P_n$-invariant space of smooth functions on $P_n$. The equivariance property \eqref{eq: k equiv cond} is immediate from
the definition of an induced representation.
%Finally, if $\sigma$ is irreducible then it follows from Theorem \ref{thm:UniSmoothIrr} that $\pr_{d,\sigma}$ has dense image.
The Proposition follows.
\end{proof}

Given a decomposition $n=m+r$ the Iwasawa decomposition on $G_{n-1}$ implies that $P_n=[(G_m \times P_r) \ltimes M_{m\times r}(F)] K_{n-1}$.
For $s=(s_1,s_2)\in\cc^2$ and $\varphi\in \pi\times \tau$ let
\[
    \varphi_s(p)=\abs{\det g_1}^{s_1}\abs{\det g_2}^{s_2}\varphi(p),\ \ \ p=[\diag(g_1,g_2)\ltimes X]k
\]
where $g_1\in G_m,\,g_2\in G_r,\,X\in M_{m\times r}(F)$ and $k\in K_{n-1}$.
It will also be convenient to denote by $I(\pi \otimes \tau,s)$ the representation of $P_n$ on
the space of $\pi \times \tau$ defined by
\[
(I(p,\pi\otimes \tau,s)\varphi)_s(x)=\varphi_s(xp), \ \ \ \varphi\in \pi\times\tau,\ p,\,x\in P_n.
\]
Thus
\[
I(\pi \otimes \tau,s)\simeq \abs{\ }^{s_1}\pi \times \abs{\ }^{s_2}\tau
\]
and the underlying space of $I(\pi \otimes \tau,s)^\infty$ is independent of $s$.
The following is straightforward from Lemma \ref{lem: sahi}.
\begin{corollary}\label{cor: IE commute}
Consider the Hilbert representations $\varrho$ of $G_r$ and $\pi$ of $G_m$ and let $s\in \cc^2$.
Then for every $j \ge0$ we have
\[
I(\pi \otimes I^jE(\varrho),s)\simeq I^jE(I(\pi \otimes \varrho,s)).
\]
\end{corollary}

\section{Representations with symplectic models}\label{sec:Sp}
\setcounter{lemma}{0}

The purpose of this section is to study linear forms invariant by the symplectic group.
We begin with a result on Speh representations that we obtain by global means.

\begin{proposition} \label{lem:SpehDist}
Let $n=2mr$, $\delta\in\widehat{G_r}$ square integrable and $\pi=U(\delta,2m)\in \udg$. Then $\Hom_{Sp(2n)}(\pi^{\infty},\cc) \ne 0$.
\end{proposition}
\begin{proof}
If $r=1$ then $\pi=\delta\circ \det$ is a character of $G_n$. The Proposition is obvious in this case.
Assume from now on that $r=2$ (and in particular that $F=\rr$).
%Since $(G_{2n},Sp(2n))$ is a Gelfand pair \cite{AS} we have $\dim \Hom _{Sp(2n)}(\varrho^{\infty},\cc) \leq 1$.
To complete the Proposition we
%need to show that $\Hom _{Sp(2n)}(\pi^{\infty},\cc)\ne 0$. For this purpose we
globalize $\pi$ to a discrete automorphic representation for which the symplectic periods have already been studied.

Let
$\Pi$ be a cuspidal automorphic representation of $GL(2,\A_\Q)$
with Archimedean component $\Pi_{\infty}\simeq\delta$. The existence of $\Pi$ is verified, for example, using the Jacquet-Langlands correspondence.
Indeed, let $D$ be the multiplicative group of the standard  quaternion algebra defined over $\qq$.
Let $\delta'$ be a representation of $D(\rr)$ associated with $\delta$ by the local Jacquet-Langlands correspondence
\cite[\S 5]{MR0401654}.
Since $\rr^*\bs D(\R)$ is compact, it is easy to construct using the trace formula an automorphic representation $\Pi'$ of $D(\A_\qq)$ so that
$\Pi'_\infty\simeq \delta'$ and $\Pi'_p$ is unramified for all primes $p>2$. It then follows from \cite[Theorem 14.4]{MR0401654} that $\pi$ is associated by the global Jacquet-Langlands correspondence to a cuspidal automorphic representation $\Pi$ of $GL(2,\A_\Q)$. In particular $\Pi_\infty \simeq \delta$ as required.

Let $\varrho$ be the unique irreducible quotient of $\abs{\det}^{\frac{m-1}2}\Pi \times \abs{\det}^{\frac{m-3}2}\Pi\times\cdots\times\abs{\det}^{\frac{1-m}2}\Pi$. It is a discrete automorphic representation
of $GL(n,\A_\Q)$ obtained by residues of
Eisenstein series (see \cite{MR1026752}).
Furthermore, its local component at infinity is $\varrho_\infty=\pi$.
Let $\varrho_\aut$ be the space of automorphic forms in (the unitary representation) $\varrho$.
Based on \cite[Theorem
3]{MR2248833}, the symplectic period defined on $\varrho_\aut$ by
\[
\ell(\phi)=  \int_{Sp(n,\Q)\bs Sp(n,\A_\Q)}\phi(h)\ dh
\]
is not identically zero. Recall that $\varrho_\aut\simeq \otimes_{p \le \infty}\tau_p$ where
$\tau_\infty=(\rho_\infty)_K$ is the $(\mathfrak{g},K)$-module of $K$-finite vectors in $\varrho_\infty$ and $\tau_p$ is the smooth part of $\varrho_p$ for $p<\infty$ \cite{MR546596}.
There is therefore an automorphic form $\phi\in \varrho_\aut$ that as a vector is of the form $\phi_\infty \otimes \phi^\infty$ with
$\phi_\infty\in \tau_\infty$ and $\phi^\infty\in \otimes_{p<\infty} \tau_p$ such that $\ell(\phi)\ne 0$.
Define
\[
\lambda(v)=\ell(v \otimes \phi^\infty),\ \ \ v \in \tau_\infty.
\]
Then $\lambda$ is a non-zero $Sp(n)\cap K$ and $\mathfrak{sp}(n)$-invariant linear
form on $\tau_\infty$ where $\mathfrak{sp}(n)$ is the Lie algebra of $Sp(n)$. By
the automatic continuity for reductive symmetric spaces  (cf.
\cite[Theorem 2.1]{MR961900} or \cite[Theorem 1]{MR1176208}) $\lambda$ extends to an $Sp(n)$-invariant linear form on the smooth part of $\rho_\infty$, i.e. it defines a non-zero element of $\Hom_{Sp(n)}(\pi^\infty,\cc)$. The Proposition follows.
\end{proof}
\begin{rem}
In \cite{GSS} an $Sp(n)$-invariant functional on the Speh representation $U(\delta,2m)$ is constructed by purely local means using \cite{MR1062967}.
\end{rem}

Next we consider induced representations.
Our main tool is a result of Carmona-Delorme that we now recall.

Let $(n_1,\dots,n_k)$ be a decomposition of $n$ and $P=MU$
the standard parabolic subgroup of $G_{2n}$ of type $(2n_1,\dots,2n_k)$ with unipotent radical $U$ and
standard Levi subgroup $M$.
Let $\j=\diag(J_{n_1},\dots,J_{n_k})$ where $J_n$ is defined by \eqref{eq: symp form} and
\[H=Sp(\j)=\{g\in G_{2n}: {}^t g \j g=\j\}.\]
Set $\tau(g)=\j{}^t g^{-1}\j^{-1}$ and let $\theta(g)={}^t g^{-1}$ be the standard Cartan involution of $G_{2n}$.
Note that $H=G^{\tau}$ and $P$ is $\theta \tau$- stable.
 Let $\sigma_i\in \ud{G_{2n_i}}$ and $0\ne \ell_i\in \Hom_{Sp(2n_i)}(\sigma_i^\infty,\cc)$, $i=1,\dots, k$. Set $\sigma=\sigma_1\otimes \cdots \otimes \sigma_k$ and $\ell=\ell_1\otimes\cdots\otimes \ell_l$.
Thus $0\ne \ell\in \Hom_{M\cap H}(\sigma^\infty,\cc)$. There is a permutation matrix $\eta\in G_{2n}$ so that ${}^t\eta \j \eta=J_n$ and therefore $\eta^{-1}Sp(\j)\eta=Sp(2n)$.
The following is therefore an application of \cite[Proposition 2 and Theorem 3]{MR1274587}.
\begin{proposition}\label{prop: CD}
With the above notation the integral
\[
\xi(\varphi; \ell,s)=\int_{(M\cap H)\bs H}\ell(\varphi_s(h\eta))\ dh,\ \ \ \varphi\in (\sigma_1 \times \cdots \times \sigma_k)^\infty
\]
converges absolutely for $\re(s_1) \gg \re(s_2) \gg \cdots \gg \re(s_k)$ and extends to a meromorphic function
of $s \in \cc^k$. Whenever holomorphic at $s$ it defines a non-zero element $\xi(\ell,s)\in \Hom_{Sp(2n)}(I(\sigma,s)^\infty,\cc)$.
\end{proposition}

\begin{theorem}\label{thm:EvenDist}
Let $\pi\in \ud{G_{2n}}$ be an even representation then $\Hom_{Sp(2n)}(\pi^{\infty},\cc)\ne 0$.
\end{theorem}
\begin{proof}
%By the multiplicity one of it is enough to show that $\Hom_{Sp(2n)}(\pi^{\infty},\cc)\ne 0$.
By the classification of the unitary dual and the recipe for the $SL(2)$-type we may write $\pi=I(\sigma,\alpha)$ where $\sigma=U(\delta_1,2m_1) \otimes \cdots \otimes U(\delta_k,2m_k)$ with $\delta_i$ square integrable and $\alpha=(\alpha_1,\dots,\alpha_k)$ with $-\frac12 <\alpha_i <\frac12$, $i=1,\dots,k$. Let $n_i$ be such that $U(\delta_i,2m_i)\in \ud{G_{2n_i}}$. By Proposition \ref{lem:SpehDist} there exists $0\ne \ell_i \in \Hom_{Sp(2n_i)}(\sigma_i^\infty,\cc)$.
By Proposition \ref{prop: CD} and using its notation we obtain a non-zero meromorphic family of linear forms
$\xi(\ell,s)\in \Hom_{Sp(2n)}(I(\sigma,s)^\infty,\cc)$. There exists a generic direction $s_0\in \cc^k$ such that
$\xi(\ell,\alpha+zs_0)$ is meromorphic in a punctured neighborhood of $z=0$ in $\cc$. Let $k_0$ be the smallest integer $k$ such that
$z^{k}\xi(\ell,\alpha+zs_0)$ is holomorphic at $z=0$. We
can now define
\[
L=\lim_{z\to 0}z^{k_0}\xi(\ell,\alpha+zs_0).
\]
Thus $0\ne L\in \Hom_{Sp(2n)}(\pi^\infty,\cc)$.

\end{proof}

\section{Proof of Theorem \ref{thm:main}}\label{sec:PfMain}

We change the setting by defining another family of Klyachko subgroups compatible with the theory of highest derivatives.
Fix a decomposition $n=2k+r$ and let
\[
H'_{2k,r}=\left\{  \left(
\begin{array}
[c]{cc}%
h & X\\
0 & u
\end{array}
\right)  \in G_n:u\in N_r,\,X\in M_{2k\times r}(F)\text{ and }h\in
Sp(2k)\right\}.
\]
Let $\phi'_{2k,r}$ be the character of  $H'_{2k,r}$ defined by
\[
\phi'_{2k,r}\left(
\begin{array}
[c]{cc}%
h & X\\
0 & u
\end{array}
\right)  =\psi_{r}(u).
\]

Let $\tau$ be the involution on $G_n$ defined by $g^\tau=w_n {}^t g^{-1} w_n$. Note that $H'_{2k,r}=H_{r,2k}^\tau$ and $\phi'_{2k,r}(h)=\phi_{r,2k}(h^\tau)$, $h\in H'_{2k,r}$.
It follows that for any $\pi\in \ud{G_n}$ we have
\[
    \Hom_{H_{r,2k}}(\pi^\infty,\phi_{r,2k})\simeq \Hom_{H'_{2k,r}}((\pi^\tau)^\infty,\phi'_{2k,r}).
\]
By the Gelfand-Kazhdan Theorem $\pi^{\tau} \simeq \widetilde{\pi}$ where $\widetilde{\pi}$ denotes the dual of $\pi$ (see e.g. \cite[Theorem 2.4.2]{MR2474319}) and therefore
\[
    \Hom_{H_{r,2k}}(\pi^\infty,\phi_{r,2k})\simeq \Hom_{H'_{2k,r}}(\tilde\pi^\infty,\phi'_{2k,r}).
\]

It further follows from the classification and the definition of the partition $\vv(\pi)$ that $\vv(\widetilde{\pi})=\vv(\pi)$ and hence $r(\widetilde{\pi})=r(\pi)$. Theorem \ref{thm:main} is therefore equivalent to the statement
\begin{equation}\label{eq: main'}
\Hom_{H'_{n-r(\pi),r(\pi)}}(\pi^\infty,\phi'_{n-r(\pi),r(\pi)})\ne0, \ \ \ \pi \in \ud{G_n}.
\end{equation}

Let $\pi\in \ud{G_n}$. If $r(\pi)=0$, i.e. $\pi$ is even, then \eqref{eq: main'} follows from Theorem \ref{thm:EvenDist}.
Assume from now on that $r=r(\pi)>0$ and let $k=(n-r)/2$. Note then that $H'_{2k,r}$ is a subgroup of $P_n$.

Write $\pi=\pi_e \times
\pi_o$ where $\pi_e\in\ud{G_{2k_1}}$ is even and
$\pi_o\in\ud{G_t}$ is odd as in Corollary \ref{cor:DecompEvenOdd}.
For $s\in \cc$ let
\[
\pi_s=I(\pi_e \otimes \pi_o,(0,s))
\]
be a representations of $G_n$ and
\[
\tau_s=I(\pi_e \otimes (\pi_o|_{P_t}),(\frac12,-2k_1+s))
\]
a representation of $P_n$.
By Corollary \ref{cor:ParInd} restriction of functions to $P_n$ is a well defined (and clearly $P_n$-equivariant) map
\[
\kappa_s: \pi_s^\infty \rightarrow \tau_s^\infty.
\]
In the parameter $s$ it is holomorphic and non-zero at each $s$.

Let $d=\text{depth}(\pi_o)$. By Corollary \ref{cor:DerOddEven} $d= r(\pi)$ and $A(\pi_o)$ is even.
Recall that $\pi_o|_{P_t}=I^{d-1}E(A(\pi_o))$ and let
\[
\sigma_s=I(\pi_e \otimes A(\pi_o),(\frac12,-2k_1+s)).
\]
By Corollary \ref{cor: IE commute} there is an isomorphism of Hilbert representations of $P_n$
\[
\tau_s \simeq I^{d-1}E(\sigma_s).
\]
Denote by
\[
\iota_s: \tau_s^\infty \rightarrow I^{d-1}E(\sigma_s)^\infty
\]
its restriction to the corresponding isomorphism between the spaces of smooth vectors.
%\Omer{Check its the correct twist!}
Thus
\[
\iota_s\circ \kappa_s:\pi_s^\infty \rightarrow I^{d-1}E(\sigma_s)^\infty
\]
is a holomorphic family of non-zero $P_n$-equivariant maps.
Let
\[
\pr_{d,\sigma_s}: I^{d-1}E(\sigma_s)^\infty \rightarrow \sigma_s^\infty
\]
be the map provided by Proposition \ref{prop: proj k sigma}. It is defined by evaluation at the identity and therefore it is independent of $s$. By  Proposition \ref{prop: proj k sigma} its restriction to the
image of $\iota_s\circ \kappa_s$ is non-zero.
Thus
\[
\pr_{d,\sigma_s}\circ \iota_s\circ \kappa_s:\pi_s^\infty \rightarrow \sigma_s^\infty
\]
is non-zero. By \eqref{eq: k equiv cond}  it is, in particular, $G_{2k}$-equivariant where $k=k_1+\frac{t-r}2$.
It follows from Theorem \ref{thm:EvenDist} together with Proposition \ref{prop: CD} that there exists a non-zero holomorphic family of linear forms
\[
\ell_s\in \Hom_{Sp(2k,F)}(\sigma_s^\infty,\cc)
\]
in a punctured disc centered at $s=0$. By possibly taking a smaller disc it further follows from Lemma \ref{lemma: gen irr} that $\sigma_s$ is irreducible in the punctured disc.
By Theorem \ref{thm:UniSmoothIrr} in this punctured disc $\pr_{d,\sigma_s}\circ \iota_s\circ \kappa_s:\pi_s^\infty \rightarrow \sigma_s^\infty$ has a dense image and therefore the holomorphic family of  linear forms
$L_s:=\ell_s \circ \pr_{d,\sigma_s}\circ \iota_s\circ \kappa_s$ on $\pi_s^\infty$ is non-zero.
By the equivariance property \eqref{eq: k equiv cond}, $L_s\in \Hom_{H_{2k,r}'}(\pi_s^\infty,\phi_{2k,r}')$.
There is therefore an integer $a$ such that $0\ne L:=\lim_{s\to 0}s^aL_s$. Thus $0\ne L\in \Hom_{H_{2k,r}'}(\pi^\infty,\phi_{2k,r}')$ and \eqref{eq: main'} follows.
This completes the proof of Theorem \ref{thm:main}.

\bibliographystyle{alpha}\def\cprime{$'$}

%\bibliography{source}

\end{document}